\documentclass[11pt,reqno]{amsart}
\usepackage{amsthm,amsmath,amssymb}
\usepackage{graphicx}
\usepackage{float}
\usepackage[colorlinks=true,
linkcolor=blue,
urlcolor=blue,
citecolor=blue]{hyperref}
\usepackage{mathtools}
\usepackage{relsize}
\usepackage{ytableau}
\usepackage{tikz}
\usepackage{enumerate}
\usepackage[toc,page]{appendix}
\usepackage{lscape}
\usepackage{multirow}
\usepackage{adjustbox,expl3,etoolbox} 
\usepackage[margin = 3.2cm]{geometry}
\usepackage{diagbox}
\usepackage{xcolor}
\usepackage{comment}
\newtheorem{thm}{Theorem}[section]
\newtheorem{lem}[thm]{Lemma}

\newtheorem{prop}[thm]{Proposition}
\newtheorem{conj}[thm]{Conjecture}
\newtheorem{cor}[thm]{Corollary}

\newtheorem{rmk}[thm]{Remark}
\newtheorem{prob}[thm]{Problem}
\newtheorem{qst}[thm]{Question}

\newtheorem{note}{Notations}
\theoremstyle{definition}

\DeclareMathOperator\sym{Sym}

\newcommand{\agl}[2]{\operatorname{AGL}_#1(#2)}

\newcommand{\psl}[2]{\operatorname{PSL}_#1(#2)}

\newcommand{\gl}[2]{\operatorname{GL}_{#1}(#2)}

\newcommand{\sln}[2]{\operatorname{SL}_{#1}(#2)}
\newcommand{\pg}[2]{\operatorname{PG}_{#1}(#2)}
\newcommand{\dih}[1]{\operatorname{D}_{#1}}

\newcommand{\itbf}[1]{{\bf{{\emph{{#1}}}}}}

\letcs\replicate{prg_replicate:nn}

\begin{document}	
	
	\title[]{On the intersection spectrum of $\psl{2}{q}$}
	
	{\author[A. Behajaina]{Angelot Behajaina\textsuperscript{1}}
		\address{Department of Mathematics, Technion - Israel Institute of Technology, Haifa, Israel}
		\thanks{\textsuperscript{1}Department of Mathematics, Technion - Israel Institute of Technology, Haifa, Israel}
		\email{abeha@campus.technion.ac.il}}
	{\author[R. Maleki]{Roghayeh Maleki\textsuperscript{2,3}}
		\address{University of Primorska, UP FAMNIT, Glagolja\v{s}ka 8, 6000 Koper, Slovenia}
		\thanks{\textsuperscript{2}University of Primorska, UP FAMNIT, Glagolja\v{s}ka 8, 6000 Koper, Slovenia}
		\thanks{\textsuperscript{3}University of Primorska, UP IAM, Muzejski trg 2, 6000 Koper, Slovenia}
		\email{rmaleki@uregina.ca}}
	
	\author[A.S. Razafimahatratra]{Andriaherimanana Sarobidy Razafimahatratra\textsuperscript{2,3,*}}
	\thanks{\textsuperscript{*} Corresponding author}
	\address{University of Primorska, UP FAMNIT, Glagolja\v{s}ka 8, 6000 Koper, Slovenia}\email{sarobidy@phystech.edu}

	\begin{abstract}
		
		Given a group $G$ and a subgroup $H \leq G$, a set $\mathcal{F}\subset G$ is called $H$\emph{-intersecting} if for any $g,g' \in \mathcal{F}$, there exists $xH \in G/H$ such that $gxH=g'xH$. The \emph{intersection density} of the action of $G$ on $G/H$ by (left) multiplication is the rational number 
		$\rho(G,H)$, equal to the maximum ratio $\frac{|\mathcal{F}|}{|H|}$, where $\mathcal{F} \subset G$ runs through all $H$-intersecting sets of $G$. The \emph{intersection spectrum} of the group $G$ is then defined to be the set
		\begin{align*}
			\sigma(G) := \left\{ \rho(G,H) : H\leq G \right\}.
		\end{align*}
		It was shown by Bardestani and Mallahi-Karai [{\it J. Algebraic
		Combin.}, 42(1):111–128, 2015] that if $\sigma(G) = \{1\}$, then $G$ is necessarily solvable. The natural question that arises is, therefore, which rational numbers larger than $1$ belong to $\sigma(G)$, whenever $G$ is non-solvable. In this paper, we study the intersection spectrum of the linear group $\psl{2}{q}$. It is shown that $2 \in \sigma\left(\psl{2}{q}\right)$, for any prime power $q\equiv 3 \pmod 4$. Moreover, when $q\equiv 1 \pmod 4$, it is proved that $\rho(\psl{2}{q},H)=1$, for any  odd index subgroup $H$ (containing $\mathbb{F}_q$) of the Borel subgroup (isomorphic to $\mathbb{F}_q\rtimes \mathbb{Z}_{\frac{q-1}{2}}$) consisting of all upper triangular matrices.
	\end{abstract}

	
	
	\date{\today}
	
	\maketitle
	
	%\tableofcontents
	
	\section{Introduction}
	
	All groups considered in this paper are finite and all group actions are left actions. Let $G$ be a group and $H\leq G$ a subgroup. The action of $G$ on the (left) cosets of $H$ by left multiplication gives rise to a transitive group whose point stabilizers are conjugate to $H$. This action of $G$ on cosets of $H$ gives rise in turn to a group homomorphism $G\to \sym(G/H)$. If the latter is injective (i.e., the action is faithful), then $G$ is realized as a permutation group of $G/H$. In fact, any finite transitive permutation group $G\leq \sym(\Omega)$ arises in this way by taking $H$ to be the point stabilizer $G_\omega$ of any $\omega \in \Omega$. A subset $\mathcal{F} \subset G$ is $H$-\itbf{intersecting} if for any $g,h\in \mathcal{F}$, there exists $xH \in G/H$ such that  $gxH = hxH$, that is, $x^{-1}h^{-1}gx \in H$.
	We will use the term \emph{intersecting} instead of $H$-intersecting if $H$ is clear from the context. The \itbf{intersection density} of an $H$-intersecting set $\mathcal{F} \subset G$ is the rational number $\rho(\mathcal{F},H) = \frac{|\mathcal{F}|}{|H|}$. Moreover, the \itbf{intersection density} of the action of $G$ on $G/H$ is 
	\begin{align}
		\rho(G,H) := \max \left\{ \rho(\mathcal{F},H) : \mathcal{F} \subset G \mbox{ is $H$-intersecting} \right\}.
	\end{align}
	 Note that, since $H$ is clearly $H$-intersecting, we have $\rho(G,H) \geq 1$. The study of the intersection density of a group and the characterization of the intersecting sets with maximum intersection density can be considered as an extension of the \itbf{Erd\H{o}s-Ko-Rado (EKR) Theorem} to permutation groups (see \cite{erdos1961intersection,godsil2016erdos} for details). In particular, we say that the action of $G$ on $G/H$ has the \emph{EKR property} if $\rho(G,H) = 1$. If $\rho(\mathcal{F},H) = \rho(G,H)$ only when $\mathcal{F}$ is a coset of $H$, then we say that $G$ has the \emph{strict-EKR property}. Transitive groups with the EKR and/or the strict-EKR property have been extensively studied in the literature  (see \cite{cameron2003intersecting,Frankl1977maximum,godsil2009new,larose2004stable}). 
	
	A notion introduced by Kutnar, Maru\v{s}i\v{c}, and Pujol \cite{kutnar2023intersection} is that of the \emph{intersection density array} of a vertex-transitive graph. Given a vertex-transitive graph $X=(V,E)$ with automorphism group $G$, the \itbf{intersection density array} of $X$ is the sequence of rational numbers $\rho(X) = [\rho_1,\rho_2,\ldots,\rho_k]$  satisfying $1\leq \rho_1<\rho_2<\ldots<\rho_k$, such that for any transitive subgroup $H\leq G$, there exists $i\in \{1,2,\ldots,k\}$ with $\rho(H,H_v) = \rho_i$, ($v\in V$). 
	
	In this paper, we introduce a similar array for the actions of $G$. Here, we choose a representation as a set to ease the notations used hereafter. The \itbf{intersection spectrum} of a group $G$ is the set of rationals 
	\begin{align}
		\sigma (G)= \left\{ \rho(G,H) : H\leq G \right\}.
	\end{align}
	
	A standard technique to study the intersection density of a group action is to transform it into a graph theoretical problem. Given two groups $H\leq G$, the \itbf{derangement graph} $\Gamma_{G,H}$ is the graph whose vertex set is $G$ and two elements $x,y\in G$ are adjacent if $x^{-1}y \not \in \cup_{g\in G} g^{-1}Hg$, i.e., it does not fix any coset in $G/H$. It is clear that an $H$-intersecting set corresponds to a \itbf{coclique} \footnote{i.e., a set of vertices in which no two of them are adjacent.} of $\Gamma_{G,H}$ and vice versa. Consequently, one obtains the identity $\rho(G,H) = \frac{\alpha(\Gamma_{G,H})}{[G:H]}$, where $\alpha(\Gamma_{G,H})$ denotes  the independence number or coclique number of $\Gamma_{G,H}$, that is, the size of the largest cocliques.
	
	Determining the intersection spectrum of an arbitrary transitive group is, of course, a complex problem. Nevertheless, some strong results are known when a restriction is imposed on the intersection spectrum. For instance, it was proved in \cite{bardestani2015erdHos} that if $\sigma(G) = \{1\}$, then $G$ is solvable. The proof used in \cite{bardestani2015erdHos} relies on the fact that every minimal finite simple group \footnote{i.e., a non-abelian simple group all of whose proper subgroups are solvable.} admits an action that does not have the EKR property, by constructing an intersecting set larger than the point stabilizer. Consequently, every non-solvable group $G$ admits an action on cosets of a subgroup $H\leq G$ such that $\rho(G,H)>1$. Thus, it is logical to ask what types of rational numbers larger than $1$ are in $\sigma(G)$, for a non-solvable group $G$. Though it is natural to ask whether one can classify the intersection spectrum of non-solvable groups, such a question is rather ambitious. Here, we note that a result of Meagher, Spiga and Tiep \cite{meagher2016erdHos} shows that any finite $2$-transitive group has intersection density $1$. Hence, if $G$ is a group admitting a (faithful) $2$-transitive action, then $1 \in \sigma(G)$.  We therefore propose the following problem.
	\begin{prob}
		Let $G$ be any finite simple group admitting a $2$-transitive action.
		\begin{enumerate}[(a)]
			\item Find $\rho\in \sigma(G)$ such that $\rho>1$.\label{probb}
			\item Is there $\rho>1$ and an infinite family $(G_i)_{i\in I}$ of groups admitting $2$-transitive representations such that $\rho \in \sigma(G_i)$, for all $i\in I$?\label{proba}
		\end{enumerate}
		\label{prob}
	\end{prob}
	
	In order to study the intersection spectrum of a group, the complete list of conjugacy classes of its subgroups is needed. 
	Groups that admit certain geometric actions are the most natural candidate, since the subgroup lattice of some of these groups are well-understood. 
	
	Our first result deals with solvable groups admitting a faithful $2$-transitive action. 
	From the classification of $2$-transitive groups, we know that a solvable $2$-transitive group admits an elementary abelian socle. We give an example of solvable groups whose intersection spectra contain an arbitrary large prime power.
	
	Let $q = p^k$ ($k \geq 1$) be a prime power and $n\geq 1$ an integer. Recall that the group $\agl{n}{q}$ is the group consisting of all affine transformations $(A,b): v \in \mathbb{F}_q^n \mapsto Av+b \in \mathbb{F}_q^n$, where $b\in \mathbb{F}_q^n$ and $A \in \gl{n}{q}$. The group operation on $\agl{n}{q}$ is given by $(A,b)(A^\prime,b^\prime) = (AA^\prime,Ab^\prime+b)$, for any $(A,b),(A^\prime,b^\prime) \in \agl{n}{q}$. 
	\begin{thm}
	If $q = p^k$ ($k \geq 1$) is a prime power, then $p^{kn-i}\in \sigma\left(\agl{n}{q}\right)$, for all $1 \leq i \leq kn$.\label{thm:solvable}
	\end{thm}	 
	
	Now, we turn our focus to the study the intersection spectrum of the group $\psl{2}{q}$, where $q$ is an odd prime power.
	In this case, since $\psl{2}{q}$ admits a doubly transitive action on the projective line $\pg{1}{q}$, we have $1\in \sigma(\psl{2}{q})$.

	In our next main result, we give an example for which Problem~\ref{prob}~\eqref{proba} is affirmative. 
	\begin{thm}
		The natural number $2$ belongs to $ \sigma(\psl{2}{q})$, for any odd prime power $q\equiv 3 \pmod 4$.\label{thm:main}
	\end{thm}
			
	In other words, we have $2 \in \bigcap_{q \equiv 3 \pmod 4} \sigma(\psl{2}{q})$. To prove this, we consider the action of $\psl{2}{q}$ on cosets of  some subgroups isomorphic to $\mathbb{Z}_{\frac{q+1}{2}}$. In the latter, the normalizer of a point stabilizer, which is  isomorphic to $\dih{q+1}$, is intersecting. The method used to prove the above theorem is the weighted version of the well-known Hoffman bound (see Lemma~\ref{lem:ratio-bound}). When $q \equiv 1 \pmod 4$, the normalizer of a point stabilizer in the action of $\psl{2}{q}$ on the  cosets of the subgroup isomorphic to $\mathbb{Z}_{\frac{q+1}{2}}$, used in the proof of Theorem~\ref{thm:main}, is no longer intersecting. As we will see later on, the fact that $-1$ is a square in $\mathbb{F}_q$ prevents the existence of an involution fixing a coset of the subgroup in question. By considering the action of $\psl{2}{q}$ on cosets of  some subgroups isomorphic to $\mathbb{Z}_{\frac{q-1}{2}}$, the fact that $q\equiv 1 \pmod 4$ implies that no involution is a derangement and the normalizer of a point stabilizer of the action, which is isomorphic to $\dih{q-1}$, is intersecting. Unfortunately, in this situation, we could not prove that a $\mathbb{Z}_{\frac{q-1}{2}}$-intersecting set in $\psl{2}{q}$ has size at most $q-1$. However, based on computational results (see Appendix~\ref{app}), we conjecture the following for $q\equiv 1 \pmod 4$.
	\begin{conj}
	We have
		\begin{align*}
			2\in \bigcap_{q \equiv 1 \pmod 4} \sigma\left(\psl{2}{q}\right).
		\end{align*}
		Moreover, this is realized by the action of $\psl{2}{q}$ on the cosets of  the subgroup isomorphic to $\mathbb{Z}_{\frac{q-1}{2}}$.
	\end{conj}
	A workaround to providing a universal element of $\sigma(\psl{2}{q})$ for $q\equiv 1\pmod 4$ (i.e., proving the above conjecture) would be to find an action of $\psl{2}{q}$ that does not have the EKR property, for each $q \equiv 1 \pmod 4$.  Therefore, we also propose the following problem, which is a reformulation of Problem~\ref{prob}~\eqref{probb}.
	\begin{prob}
		For any $q\equiv 1 \pmod 4$, find $H\leq \psl{2}{q}$ such that $\rho\left(\psl{2}{q},H\right)>1$.\label{prob2}
	\end{prob}

	In \cite[Theorem~6.1]{hujdurovic2022intersection-cyclic}, the authors showed that $\frac{4}{3} \in \sigma(\psl{2}{q})$, when $q \equiv 1 \pmod{12}$ is not a power of $5$. In \cite[Conjecture~5.4]{meagher2022intersection}, it is conjectured that if $q\equiv 1\pmod 4$ and $q^2\equiv 1\pmod 5$, then $\frac{4}{3} \in \sigma\left(\psl{2}{q}\right)$. In all these cases, the actions of $\psl{2}{q}$ are on a cyclic group of order $3$.
	
	Since the conjugacy classes of subgroups of $\psl{2}{q}$ are known \cite[Theorem~2.1]{king2005subgroup}, we can do an exhaustive analysis of the intersection density of all actions of $\psl{2}{q}$; there are 22 families of such  conjugacy classes. However, it turns out that even for small values of $q$, the intersection density of the actions $\psl{2}{q}$ are computationally hard to obtain. In general, the known techniques available in the literature to compute the intersection density all fall short, in the several cases that we considered. Therefore, we instead give a family of subgroups $H\leq\psl{2}{q}$ that are ruled out for Problem~\ref{prob2}.		
	For any odd prime power $q$, the Borel subgroup $\mathbb{B}_q$ is the subgroup consisting of upper triangular matrices, modulo $\{\pm I\}$ ($I$ is the identity of $\sln{2}{q}$), in $\psl{2}{q}$. Clearly, we have $|\mathbb{B}_q| = \frac{q(q-1)}{2}$. For any odd $r \mid \frac{q-1}{2}$, there is a unique subgroup $\mathbb{M}_r \leq \mathbb{B}_q$ (see Notations~\ref{note:1e}), which is isomorphic to  $\mathbb{F}_q \rtimes \mathbb{Z}_{\frac{q-1}{2r}}$. We prove the following theorem.

	\begin{thm}\label{thm:secon}
		Let $q\equiv 1 \pmod 4 $ be an {odd} prime power. If $r\mid \frac{(q-1)}{2}$ is odd, then $\rho\left(\psl{2}{q},\mathbb{M}_r\right) = 1$.\label{thm2}
	\end{thm}		

		When $r=1$, then we recover the result in \cite{meagher2011erdHos} about the intersection density of $\psl{2}{q}$ in its natural action on $\pg{1}{q}$. In fact, we can show a statement that is slightly stronger than just the one given in Theorem~\ref{thm2} (see Corollary \ref{rem:pluthm6}). This, in turn, is a step in the characterization of the largest intersecting sets in the action of $\psl{2}{q}$ on cosets of  $\mathbb{M}_r$.
				
		This paper is organized as follows. In Section~\ref{sect:background}, we give all necessary background results needed in the paper. In Section~\ref{sect:thm-abelian}, we prove Theorem~\ref{thm:solvable}. In Section~\ref{sect:cc} and Section~\ref{sect:rep}, we give the necessary background about the representation theory of $\psl{2}{q}$. Section~\ref{sect:proof-3mod4} and Section~\ref{sect:thm-1mod4} are devoted to the proof of Theorem~\ref{thm:main} and Theorem~\ref{thm2}, respectively.
			
			\section{Background results}\label{sect:background}
			
			We recall some facts about association schemes and spectral graph theory techniques that are needed throughout the proofs.
			
			\subsection{The ratio bound}	
			For a group $G$ and a subgroup $H\leq G$, recall that  $\Gamma_{G,H}$ is the \emph{derangement graph} of the action of $G$ on cosets of $H$. A derangement in this action is an element of $G$ that does not fix any coset of $H$, or equivalently, an element that is not conjugate to an element of $H$. It is easy to see that the derangement graph of a transitive group $G$ is regular with valency equal to the number of derangements. In fact, the left-regular representation $\operatorname{R}(G) = \left\{ \rho_g: x \in G \mapsto gx \in G :\ g\in G \right\}$ is in the automorphism group of $\Gamma_{G,H}$. Further, $\Gamma_{G,H}$ is a \emph{Cayley graph} whose connection set is the set of all derangements. 
			
			One important aspect of the derangement graph of a group action is that it encodes the intersecting sets of the group. Recall that a \emph{coclique} in a graph is a subset of vertices that are pairwise non-adjacent; we will denote the maximum size of cocliques in a graph $X = (V,E)$ by $\alpha(X)$. For any group $G$ and $H\leq G$, a subset $S\subset G$ is $H$-intersecting if and only if it is a coclique in the derangement graph $\Gamma_{G,H}$. Therefore, we can reformulate the intersection density of the action of $G$ to be
			\begin{align*}
				\rho(G,H) = \frac{\alpha(\Gamma_{G,H})}{|H|}. 
			\end{align*}
		Hence, we may use all graph theoretical tools at our disposal to find the invariant $\alpha(\Gamma_{G,H})$. One of the most powerful tools to compute $\alpha(\Gamma_{G,H})$ is via a weighted adjacency matrix.
			
			For the remainder of this subsection, we let $G$ be a group and $H\leq G$ a subgroup. Recall that a symmetric matrix $B$ indexed by the group elements in $G$ in its rows and columns is called a \itbf{weighted adjacency matrix} of $\Gamma_{G,H}$ if $B(g,h) = 0$, for any $g$ and $h$ that are non-adjacent in $\Gamma_{G,H}$. For any subset $S\subset G$, the \emph{characteristic vector} of $S$ is the vector $v_S$ indexed in its rows by elements of $G$, and such that $v_S(x)$ is equal to $1$ if $x \in S$ and $0$ otherwise. We also let $\mathbf{1}_{|G|}$ denote the vector of $\mathbb{Z}^{|G|}$ whose entries consist of $1$.
			
			Recall the following theorem which is a generalization of the Hoffman bound (see \cite{godsil2016erdos} for the proof).	
			\begin{thm}
				If $B$ is a weighted adjacency matrix of $\Gamma_{G,H}$ with maximum eigenvalue $d$ and minimum eigenvalue $\tau$, then
				\begin{align*}
					\alpha ( \Gamma_{G,H}) \leq \frac{|G|}{1-\frac{d}{\tau}}.
				\end{align*}
				Moreover, if equality holds for some coclique $S$ of $\Gamma_{G,H}$, then $v_S - \frac{|S|}{|G|}\mathbf{1}_{|G|}$ is an eigenvector of $B$ for the smallest eigenvalue $\tau$.\label{lem:ratio-bound}
			\end{thm}
			
			The upper bound given in the above lemma requires the knowledge of the extremal (the largest and the least) eigenvalues of the weighted adjacency matrix. Therefore, it is in our interest to find a weighted adjacency matrix whose extremal eigenvalues are easy to compute. In the next subsection, we will give a classical method to find a weighted adjacency matrix whose eigenvalues are fairly easy to compute.
			
			\subsection{The conjugacy class association scheme}
			Let $\mathcal{B} = \{A_0,A_1,\ldots,A_{k-1}\}$ be a set of $n\times n$ $01$-matrices. We say that $\mathcal{B}$ is an \itbf{association scheme} if it satisfies
			\begin{enumerate}[(1)]
				\item $A_0 = I_n$, where $I_n$ is the $n\times n$ identity matrix,
				\item $A_0+A_1+ \ldots + A_{k-1} = J_n$, where $J_n$ is the $n\times n$ matrix whose entries consist of $1$,
				\item for any $i \in \{0,1,\ldots,k-1\}$, we have $A_i^T \in \mathcal{B}$,
				\item for any $i,j \in \{0,1,\ldots,k-1\}$, we have 
				\begin{align*}
					A_i A_j &= \sum_{t = 0}^{k-1} p_{ij}^t A_t,
				\end{align*} 
				for some non-negative integers $(p_{ij}^t)$.
			\end{enumerate}		
			Given a finite group $G$ with conjugacy classes $C_0=\{1\},C_1,\ldots,C_{k-1}$, we define the $|G| \times |G|$ matrices $(A_i)_{i= 0,1,\ldots,k-1}$ which are indexed by $G$ in their rows and columns, such that
			\begin{align*}
				A_i(g,h) = 
				\begin{cases}
					1 & \mbox{ if } h^{-1}g \in C_i,\\
					0 & \mbox{ otherwise,}
				\end{cases}
			\end{align*}
			for any $i\in \{0,1,\ldots,k-1\}$ and $g,h \in G$. The set $\mathcal{B}(G) = \{ A_i : i\in \{0,1,\ldots,k-1\} \}$ forms an association scheme called the \itbf{conjugacy class (association) scheme} of $G$.
			
			It is not hard to see that a matrix in $\mathcal{B}(G)$ is in fact the adjacency matrix of the Cayley digraph of $G$ whose connection set is the corresponding conjugacy class. Hence, the adjacency matrix of the derangement graph $\Gamma_{G,H}$ is a sum of matrices in $\mathcal{B}(G)$. In general, a (symmetric) linear combination of matrices in $\mathcal{B}(G)$, corresponding to conjugacy classes of derangements, is a weighted adjacency matrix of the derangement graph $\Gamma_{G,H}$.
			
			We end this section by recalling a result from \cite{babai1979spectra} about the eigenvalues of these types of weighted adjacency matrix. We will formulate this result in terms of the derangement graph, for convenience.
			
			\begin{thm}
				Let $G$ be a  group and $H\leq G$. Assume that $C_1,C_2,\ldots,C_t$ ($t \geq 1$) are conjugacy classes of derangements for the action of $G$ on the cosets of $H$ and let $A_1,A_2,\ldots,A_t$ be the corresponding matrices in $\mathcal{B}(G)$.  Let $(\alpha_i)_{i = 1,2,\ldots,t} \subset  \mathbb{R}$ be any sequence such that $\mathcal{A}= \alpha_1 A_1+ \alpha_2 A_2+ \ldots + \alpha_t A_t$ is symmetric, that is, $\mathcal{A}$ is a weighted adjacency matrix of $\Gamma_{G,H}$. Then its eigenvalues are of the form
				\begin{align*}
					\lambda_\chi = \frac{1}{\chi(1)}\sum_{i = 1}^t \alpha_i \sum_{g \in C_i} \chi(g),
				\end{align*}
				where $\chi$ runs through all (complex) irreducible characters of $G$.

				For any irreducible character $\chi$ of $G$, let $V_\chi$ be the corresponding regular $\mathbb{C}G$-module. Then, the eigenspace of $\mathcal{A}$ corresponding to the eigenvalue $\lambda$ is the $\mathbb{C}G$-module
				\begin{align*}
					U_\lambda:=\bigoplus_{\chi} V_\chi,
				\end{align*}
				where the sum runs through all irreducible characters $\chi$ such that $\lambda_\chi =\lambda$.
				\label{thm:evalues}
			\end{thm}
			
			%%%%%%%%%%%%%%%%%%%%%%%%%%%%%%%%%%%%%%%%%%%%%%%%%%%%%%%%%%%%%%%%%%%%%%%%%%%%%%%%%%%%%%			
			\section{Proof of Theorem~\ref{thm:solvable}}\label{sect:thm-abelian}
			
		Let $q=p^k$ ($k \geq 1$) be a prime power and let $1\leq i\leq kn$. Let $H_i$ be a subgroup of $\mathbb{F}_q^n$ of order $p^i$. Consider the subgroup $\mathbf{E}_i= \left(I,H_i\right)$ of $\agl{n}{q}$ and its derangement graph $\Gamma_{n,q} = \Gamma_{\agl{n}{q},\mathbf{E}_i}$.

				First, we prove that $\Gamma_{n,q}$ has a clique of order $\prod_{j=0}^{n-1} (q^n-q^j)$. Then, we prove that there is a coclique attaining the bound given by the clique-coclique bound \cite[Corollary 2.1.2]{godsil2016erdos}. 
				
				Let us prove that the subgroup $\gl{n}{q}$ of $\agl{n}{q}$ is indeed a clique in $\Gamma_{n,q}$. Let $A,B \in \gl{n}{q}$ be distinct. Then,
				\begin{align*}
					(A,0)^{-1} (B,0) = (A^{-1},0)(B,0) = (A^{-1}B,0)
				\end{align*}
				is a derangement since $(A^{-1}B,0)$ cannot be in any conjugate of $\mathbf{E}_i$. Therefore, $\gl{n}{q}$ is a clique in $\Gamma_{n,q}$. By the clique-coclique bound, we have
				\begin{align}
					\alpha(\Gamma_{n,q}) \leq \frac{|\agl{n}{q}|}{|\gl{n}{q}|} = q^n.\label{eq:bound-agl}
				\end{align}
				Next, we prove that $\mathbb{F}_q^n$ is a coclique in $\Gamma_{n,q}$. Let $\mathbf{u}$ and $\mathbf{v}$ be two elements of $\mathbb{F}_q^n$ and let $A \in \gl{n}{q}$ such that $A(\mathbf{v}-\mathbf{u}) \in H_i$. Since $(A,0)(I,\mathbf{u})^{-1} (I,\mathbf{v})(A^{-1},0) = (I,A(\mathbf{v}-\mathbf{u})) \in \mathbf{E}_i$, the permutations $(I,\mathbf{u})$ and $(I,\mathbf{v})$ are intersecting. Therefore, $\mathbb{F}_q^n$ is intersecting. Combining this with the bound in \eqref{eq:bound-agl}, we have
				\begin{align*}
					\alpha(\Gamma_{n,q}) = q^n.
				\end{align*}
				Consequently, the intersection density of the group $\agl{n}{q}$ acting on cosets of $\mathbf{E_i}$ is equal to $\frac{q^n}{p^i}$, which completes the proof.
			%%%%%%%%%%%%%%%%%%%%%%%%%%%%%%%%%%%%%%%%%%%%%%%%%%%%%%%%%%%%%%%%%%%%%%%%%%%%%%%%%%%%%%%%%%%%%%%%%%%%%%%		

			\section{Conjugacy classes of $\psl{2}{q}$}\label{sect:cc}
			In this section, we recall the conjugacy classes of $\psl{2}{q}$, for any odd prime power $q$, as computed in \cite{adams2002character}. Moreover, we establish some basic results that will be used later. 
			
Let $\omega$ be a primitive element of $\mathbb{F}_q$, i.e., $\mathbb{F}_q^* = \langle \omega\rangle$. 
			\begin{comment}
				Let $f(t) = t^2+bt +c$ be an irreducible polynomial over $\mathbb{F}_q$ and let $\Delta = b^2-4c$. 
			\end{comment} 
			Consider $\Delta \in \mathbb{F}_q \setminus (\mathbb{F}_q)^2$ and let $\delta= \sqrt{\Delta}$. For convenience, when $q \equiv 3 \pmod{4}$, we take $\Delta=-1$. Then, $\mathbb{F}_{q^2} = \mathbb{F}_q(\delta)$ is a quadratic extension of $\mathbb{F}_q$. If $q \equiv 1 \pmod 4$, then we let $\sqrt{-1} = \omega^{\frac{q-1}{4}}$. 
			
			\subsection{When $q \equiv 1 \pmod 4$}\label{sec:conjq14}  There are $\frac{q+5}{2}$ conjugacy classes of $\psl{2}{q}$: 
			\begin{enumerate}[(T1)]
				\item The conjugacy class consisting only of the identity matrix $\{\overline{I}\}$.
				\item The conjugacy class consisting of all matrices conjugate to $$c_2(1):=\overline{\begin{bmatrix}
						1 & 1\\
						0 &  1
				\end{bmatrix}}.$$
				\noindent We let $C_2(1)$ be the corresponding matrix in the conjugacy class scheme (of $\psl{2}{q}$).
				\item The conjugacy class consisting of all matrices conjugate to $$c_2(\Delta):=
				\overline{		
					\begin{bmatrix}
						1 & \Delta\\
						0 &  1
				\end{bmatrix}}.$$
				\noindent We let $C_2(\Delta)$ be the corresponding matrix in the conjugacy class scheme.
				\item  Note that \label{type4}
				$ \left \langle\overline{\begin{bmatrix}
						\omega & 0 \\
						0 & \omega^{-1}
				\end{bmatrix}} \right \rangle = \left \langle\begin{bmatrix}
					\omega & 0 \\
					0 & \omega^{-1}
				\end{bmatrix} \right \rangle/\{ \pm I \}$ has order $\frac{q-1}{2}$. For any $x \in \mathbb{F}_q \setminus \{0, \pm 1\}$, we let 
				$$
				c_{3}(x)=\overline{\begin{bmatrix}
						x & 0 \\
						0 & x^{-1}
				\end{bmatrix}}.		
				$$
				For each $x \in \mathbb{F}_q \setminus \{0,\pm 1\}$, the matrices $c_3(x)$, $c_3(x^{-1})$, $c_3(-x)$ and $c_3(-x^{-1})$ are in the same conjugacy class. The next type of conjugacy class consists of matrices conjugate to 
				\begin{align*}
					c_3(\omega^i):= \overline{\begin{bmatrix}
							\omega^i & 0 \\
							0 & \omega^{-i}
					\end{bmatrix}},
				\end{align*}
				for $i \in  \left\{1,\ldots,\frac{q-5}{4}  \right\}$. There are $\frac{q-5}{4}$ conjugacy classes of this type.
				
				\noindent We let $C_3(\omega^i)$ be the matrix in the conjugacy class scheme corresponding to $c_3(\omega^i)$, for any $i\in \{1,\ldots,\frac{q-5}{4}\}$.
				\item The conjugacy class consisting of matrices conjugate to 
				
				\begin{align*}
					c_3\left(\sqrt{-1}\right)= c_3\left(\omega^{\tfrac{q-1}{4}}\right)= \overline{\begin{bmatrix}
							\sqrt{-1} & 0\\
							0 & -\sqrt{-1}
					\end{bmatrix}}.
				\end{align*}
					\noindent We let $C_3(\sqrt{-1})$ be the matrix in the conjugacy class scheme corresponding to $c_3(\sqrt{-1})$.				
				\item The field $\mathbb{F}_{q^2}=\mathbb{F}_q(\delta)$ can be embedded into the ring $\operatorname{M}_2(q)$ of all $2 \times 2$ matrices with entries in $\mathbb{F}_q$, via \label{type6}
				\begin{align*}
					\mathbb{F}_{q^2} &\to \operatorname{M}_2(q)\\
					a+\delta b&\mapsto
					\begin{bmatrix}
						a & b\Delta\\
						b & a\\ 
					\end{bmatrix}.
				\end{align*} 
				This ring homomorphism induces a group homomorphism from $\mathbb{F}_{q^2}^*$ to $\gl{2}{q}$. Consider the norm map $N: z \in \mathbb{F}_{q^2}^* = \left(\mathbb{F}_q(\delta)\right)\setminus \{0\} \mapsto N(z)=z^{q+1} \in \mathbb{F}_q^*$ and let $\operatorname{E}_q$ be its kernel. The last type of conjugacy classes of $\psl{2}{q}$ consists of all matrices conjugate to
				\begin{align*}
					c_4 (z) := 
					\overline{
						\begin{bmatrix}
							a & \Delta b\\
							b &  a
					\end{bmatrix}},
				\end{align*}
				where $z = a+\delta b \in \operatorname{E}_q \setminus \{ \pm 1 \}$. Note that, for any $z \in E_q \setminus \{\pm 1\}$, the matrices $c_4(z) ,\ c_4(z^{-1}) ,\ c_4(-z) ,\mbox{ and } c_4(-z^{-1})$ are in the same conjugacy class. There are $\frac{q-1}{4}$ conjugacy classes of this type.
				
				\noindent For any $z\in E_q \setminus \{\pm 1\}$, we denote by $C_4(z)$ the matrix in the conjugacy class scheme corresponding to $c_4(z)$.
			\end{enumerate}

			\subsection{When $q \equiv 3 \pmod{4}$}\label{subsec:q34}
			We also recall the conjugacy classes of $\psl{2}{q}$ (see \cite{adams2002character} for details). The assumption on $q$ implies that $\sqrt{-1} \notin \mathbb{F}_q$. Hence, we have $\mathbb{F}_{q^2}=\mathbb{F}_q(\sqrt{-1})$.  There are $\frac{q+5}{2}$ conjugacy classes of $\psl{2}{q}$ in this case.
			\begin{enumerate}[(U1)]
				\item The conjugacy class consisting only of the identity matrix $\{\overline{I}\}$.
				\item The conjugacy class consisting of all matrices conjugate to
				$$
				c_2'(1)=\overline{
					\begin{bmatrix}
						1 & 1 \\
						0 & 1
					\end{bmatrix}
				}.$$
				\noindent We let $C_2^\prime(1)$ be the matrix in the conjugacy class scheme corresponding to $c_2^\prime(1)$.
				\item The conjugacy class consisting of all matrices conjugate to
				$$
				c'_2(-1)=
				\overline{
					\begin{bmatrix}
						1 & -1\\
						0 & 1
					\end{bmatrix}
				}.
				$$
					\noindent We let $C_2^\prime(-1)$ be the matrix in the conjugacy class scheme corresponding to $c_2^\prime(-1)$.
				\item
				Recall that  $\overline{\begin{bmatrix}
						\omega &0 \\
						0 & \omega^{-1}
				\end{bmatrix}}$ has order $\frac{q-1}{2}$.  For any $x \in \mathbb{F}_q \setminus \{0, \pm 1\}$, we let 
				$$
				c'_{3}(x)=\overline{\begin{bmatrix}
						x & 0 \\
						0 & x^{-1}
				\end{bmatrix}}.		
				$$
				We note that, for $x \in \mathbb{F}_q \setminus \{0, \pm 1\}$, the matrices $c'_3(x)$, $c'_3(x^{-1})$, $c'_3(-x)$, and $c'_3(-x^{-1})$ are in the same conjugacy class. The next type of conjugacy class consists of matrices conjugate to 
				\begin{align*}
					c'_3(\omega^i):= \overline{\begin{bmatrix}
							\omega^i & 0 \\
							0 & \omega^{-i}
					\end{bmatrix}},
				\end{align*}
				for $i \in  \left\{1,\ldots,\frac{q-3}{4}  \right\}$. There are $\frac{q-3}{4}$ conjugacy classes of this type. 	\noindent We let $C_3^\prime(\omega^i)$ be the matrix in the conjugacy class scheme corresponding to $c_3^\prime(\omega^i)$, for any $i\in \{1,\ldots,\frac{q-3}{4}\}$.
				\item As before consider the injective group homomorphism
				\begin{align}
					\begin{split}
						\mathbb{F}_{q^2}^* &\to \gl{2}{q}\\
						a+ b\sqrt{-1}&\mapsto
						\begin{bmatrix}
							a & -b\\
							b & a\\ 
						\end{bmatrix}.
					\end{split}\label{norm-map} 
				\end{align}
				Also, consider the norm map $N: z \in \mathbb{F}_{q^2}^* = \left(\mathbb{F}_q(\sqrt{-1})\right)\setminus \{0\} \mapsto N(z)=z^{q+1} \in \mathbb{F}_q^*$. Let $\operatorname{E}_q$ be the kernel of $N$. The last type of conjugacy classes of $\psl{2}{q}$ consists of all matrices conjugate to
				\begin{align*}
					c'_4 (z) := 
					\overline{
						\begin{bmatrix}
							a & - b\\
							b &  a
					\end{bmatrix}},
				\end{align*}
				where $z = a+ b \sqrt{-1} \in \operatorname{E}_q \setminus \{ \pm 1 \}$. Note that, for any $z \in E_q \setminus \{\pm 1\}$, the matrices $c_4^\prime(z) ,\ c_4^\prime(z^{-1}) ,\ c_4^\prime(-z) ,\mbox{ and } c_4^\prime(-z^{-1})$ are again in the same conjugacy class. There are ${\frac{q+1}{4}}$ conjugacy classes of this type. 	\noindent We let $C_4^\prime(z)$ be the matrix in the conjugacy class scheme corresponding to $c_4^\prime(z)$, for any $z \in E_q\setminus\{\pm 1\}$.
			\end{enumerate}

			\section{Representation theory of $\psl{2}{q}$}\label{sect:rep}
			
			In this section, we give the necessary results about the representation theory of $\psl{2}{q}$ that are needed for this work. 
			
			\subsection{Basic representation theory} For this part, we refer to \cite{fulton2013representation} for more details.
			Now let $G$ be a group. A \emph{(complex) representation} of $G$ is a group homomorphism $\mathfrak{X}: G \to \gl{n}{\mathbb{C}}$, where $n \geq 1$ is a positive integer, called the \emph{dimension} of $\mathfrak{X}$. To any such $\mathfrak{X}$, we can associate another map $\chi: g \in G \mapsto \operatorname{Trace}(\mathfrak{X}(g)) \in \mathbb{C}$, called the \emph{character} of $G$ afforded by $\mathfrak{X}$  of {\it degree} $n$. Moreover, one obtains a module action of the group algebra $\mathbb{C}G$ on the vector space $\mathbb{C}^n$ via $g.v = \mathfrak{X}(g)v$, for any $g\in G$ and $v \in \mathbb{C}^n$. Consequently, $\mathbb{C}^n$ is a $\mathbb{C}G$-module. Conversely,  any $\mathbb{C}G$-module $V$ of finite dimension over $\mathbb{C}$ gives rise to a representation of $G$, by writing the action of $g \in G$ as a matrix in a fixed basis of $V$. We say that $\mathfrak{X}$ is \emph{irreducible} if its corresponding $\mathbb{C}G$-module $\mathbb{C}^n$ is irreducible. In this case, the character $\chi$ afforded by $\mathfrak{X}$ is called \emph{irreducible}. Throughout the paper, we denote by $\operatorname{Irr}(G)$ the complete set of all irreducible characters of $G$. It is well-known that any character of $G$ is a sum of some irreducible characters. In other words, for any character $\psi$ of $G$, there exists $\left(m_\chi\right)_{\chi \in \operatorname{Irr}(G)} \subset \mathbb{N}$ such that
			\begin{align}
				\psi = \sum_{\chi \in \operatorname{Irr}(G)} m_\chi\chi.\label{eq:char-decomp}
			\end{align}
		 For any $\chi \in \operatorname{Irr}(G)$, the integer $m_\chi$ is called the {\it multiplicity} of $\chi$ in $\psi$. If $m_\chi >0$, then $\chi$ is called a \emph{constituent} of $\psi$. The space of all class functions \footnote{i.e., complex-valued maps that are constant on any conjugacy class} of a group is endowed with the natural inner product defined by
			\begin{align}
				\langle \psi,\phi \rangle_G = \frac{1}{|G|} \sum_{g\in G} \psi(g) \overline{\phi(g)},
			\end{align}
			for any two class functions $\psi$ and $\phi$. The multiplicity of an irreducible character is easily computed using this inner product. In fact, we have $m_\chi =  \langle \psi, \chi \rangle$ in \eqref{eq:char-decomp}, for any $\chi \in \operatorname{Irr}(G)$. Moreover $\psi$ is irreducible if and only if $\langle  \psi ,\chi \rangle = \delta_{\psi,\chi}$, for any $\chi \in \operatorname{Irr}(G)$.
			
			Assume further now that $H$ is a subgroup of $G$. The action of $G$ on $\Omega = G/H$ yields a group homomorphism $G \to \sym(\Omega)$, which naturally extends to  a representation $G \rightarrow \gl{|\Omega|}{\mathbb{C}}$ of dimension $|\Omega|$. The latter associate any $g\in G$ to the permutation matrix $P_g \in \gl{|\Omega|}{\mathbb{C}}$, whose $(xH,yH)$-entry is equal to $1$ if $gxH = yH$, and $0$ otherwise. The character $\mathbf{fix}_H^G$ corresponding to this representation is called the \itbf{permutation character} of $G$ in its action on $\Omega = G/H$. 
			It is worth noting that $\mathbf{fix}_H^G(g)$ is equal to the number of cosets of $H$ fixed by $g$, for any $g\in G$.  
			This character  encodes some combinatorial invariants associated to $G$. For instance, if $\chi_0$ denotes the trivial character of $G$ \footnote{i.e., the character equal to the trivial homomorphism from $G$ to $\mathbb{C}^\times$.}, then
			\begin{align}
				\langle \mathbf{fix}_H^G,\chi_0 \rangle_G = \frac{1}{|G|} \sum_{g\in G} \mathbf{fix}_H^G(g),
			\end{align}
			which, by the orbit-counting lemma, is the number of orbits of the action $G$ on $\Omega$. 

			\subsection{The character tables of $\psl{2}{q}$}\label{sec:charpsl22}
			As in the previous section, the irreducible characters of $\psl{2}{q}$ depend on the congruence of $q$ modulo $4$. Here we recall these irreducible characters, as given in \cite{adams2002character}.
			Also, we refer to \cite{adams2002character} for the notations. The irreducible characters are one of the following:
			\begin{itemize}
				\item $\rho^\prime(1)$.
				\item $\overline{\rho}(1)$.
				\item $\rho(\alpha)$, where $\alpha \in \operatorname{Irr}(\mathbb{F}_q^*)$  such that $\alpha^2 \neq 1$ and $\alpha(-1)=1$ (note that $\rho(\alpha) = \rho(\alpha^{-1})$).\\
				\item $\pi(\chi)$, where $\chi \in \operatorname{Irr}(\operatorname{E}_q)$ such that $\chi^2 \neq 1$,  $\chi(-1)=1$ and $\chi \neq \overline{\chi}$ (note that $\pi(\chi) \simeq \pi(\overline{\chi})$).
				\item $\omega_e^\pm $, if $q \equiv 1 \pmod 4$.
				\item $\omega_0^{\pm}$, if $q \equiv 3 \pmod 4$.
			\end{itemize}
			We give the character table of $\psl{2}{q}$, for $q \equiv 1 \pmod{4}$ and $q \equiv 3 \pmod 4$, in Table~\ref{tab:char-table} and Table~\ref{tab:q-cong2}, respectively. Here, $\zeta \in \operatorname{Irr}(\mathbb{F}_q^*)$ denotes the unique irreducible character of $\mathbb{F}_q^*$ such that $\zeta(\omega) = -1$.
			{\footnotesize
				\begin{table}[t]
					\begin{tabular}{|c|c|c|c|c|c|c|c|}
						\hline
						&Number&$1$&$1$ &$1$ &$\frac{q-5}{4}$&$1$&$\frac{q-1}{4}$\\
						\hline
						&Size&$1$&$\frac{q^2-1}{2}$& $\frac{q^2-1}{2}$ & $q(q+1)$& $\frac{q(q+1)}{2}$& $q(q-1)$ \\ \hline
						Character&Number& $\overline{\begin{bmatrix}
								1 &0\\
								0 & 1
						\end{bmatrix}}$ & $\overline{\begin{bmatrix}
								1 & 1\\
								0 & 1
						\end{bmatrix}}$ & $\overline{\begin{bmatrix}
								1 & \Delta \\
								0 & 1
						\end{bmatrix}}$& $\overline{\begin{bmatrix}
								x & 0 \\
								0 & x^{-1}
						\end{bmatrix}}$& $\overline{\begin{bmatrix}
								\sqrt{-1} &0 \\
								0 & \left(\sqrt{-1}\right)^{-1}
						\end{bmatrix}}$& $\overline{\begin{bmatrix}
								x & \Delta y\\
								y & x
						\end{bmatrix}}$ \\ \hline
						$\rho(\alpha)$ & $\frac{q-5}{4}$ & $(q+1)$ & $1$ & $1$&$\alpha(x)+\alpha(x^{-1})$ & $2\alpha(\sqrt{-1})$ & $0$ \\ \hline
						$\overline{\rho}(1)$ & $1$& $q$ & $0$ & $0$ & $1$ & $1$ &$-1$\\ \hline
						$\rho^\prime(1)$&$1$&$1$&$1$&$1$&$1$&$1$&$1$\\ \hline
						$\pi(\chi)$ & $\frac{q-1}{4}$ &$q-1$& $-1$ & $-1$ & $0$ & $0$ & $-\chi(z) - \chi(z^{-1})$ \\ \hline
						$\omega_e^{\pm}$ & $2$ & $\frac{q+1}{2}$ & $\omega_e^{\pm}(1,1)$ & $\omega_e^\pm (1,\Delta)$ & $\zeta(x)$ & $\zeta(\sqrt{-1})$ & $0$ \\ \hline
					\end{tabular}
					\caption{Character table of $\psl{2}{q}$, for $q\equiv 1 \pmod 4$.}\label{tab:char-table}
			\end{table}}
			
			{\footnotesize 
				\begin{table}[b]
					\begin{tabular}{|c|c|c|c|c|c|c|c|}
						\hline
						&Number&$1$&$1$ &$1$ &$\frac{q-3}{4}$&$1$&$\frac{q-3}{4}$\\
						\hline
						&Size&$1$&$\frac{q^2-1}{2}$& $\frac{q^2-1}{2}$ & $q(q+1)$& $\frac{q(q-1)}{2}$& $q(q-1)$ \\ \hline
						Character&Number& $\overline{\begin{bmatrix}
								1 &0\\
								0 & 1
						\end{bmatrix}}$ & $\overline{\begin{bmatrix}
								1 & 1\\
								0 & 1
						\end{bmatrix}}$ & $\overline{\begin{bmatrix}
								1 & -1 \\
								0 & 1
						\end{bmatrix}}$& $\overline{\begin{bmatrix}
								x & 0 \\
								0 & x^{-1}
						\end{bmatrix}}$& $\overline{\begin{bmatrix}
								0 &-1 \\
								1 & 0
						\end{bmatrix}}$& $\overline{\begin{bmatrix}
								x & - y\\
								y & x
						\end{bmatrix}}$ \\ \hline
						$\rho(\alpha)$ & $\frac{q-3}{4}$ & $(q+1)$ & $1$ & $1$&$\alpha(x)+\alpha(x^{-1})$ & $0$ & $0$ \\ \hline
						$\overline{\rho}(1)$ & $1$& $q$ & $0$ & $0$ & $1$ & $-1$ &$-1$\\ \hline
						$\rho^\prime(1)$&$1$&$1$&$1$&$1$&$1$&$1$&$1$\\ \hline
						$\pi(\chi)$ & $\frac{q-3}{4}$ &$q-1$& $-1$ & $-1$ & $0$ & $-2\chi(\sqrt{-1})$ & $-\chi(z) - \chi(z^{-1})$ \\ \hline
						$\omega_0^{\pm}$ & $2$ & $\frac{q-1}{2}$ & $\omega_0^{\pm}(1,1)$ & $\omega_0^\pm (1,-1)$ & $0$ & $-\chi_0(\sqrt{-1})$ & $-\chi_0(z)$ \\ \hline
					\end{tabular}
					\caption{Character table $\psl{2}{q}$, for $q\equiv 3 \pmod 4$.}\label{tab:q-cong2}
				\end{table}
			}

			\section{Proof of Theorem~\ref{thm:main} for $q\equiv 3 \pmod 4$}\label{sect:proof-3mod4}
			
		Let $q \equiv 3 \pmod{4}$. Recall that the kernel $\operatorname{E}_q \leq \mathbb{F}_{q^2}^*$ of the norm map (see \S\ref{subsec:q34}) is a cyclic group of order $q+1$. Further, $\operatorname{E}_q$ can be embedded into $\sln{2}{q}$ via the map given in \eqref{norm-map}, which induces an embedding of $E_q/\{\pm 1\}$ into $\psl{2}{q}$. Fix $\varepsilon \in \operatorname{E}_q$ such that $\operatorname{E}_q = \langle \varepsilon \rangle$, and let $A_\varepsilon \in \psl{2}{q}$ be the image of $\varepsilon\{\pm 1\}$ with respect to this embedding.
			
			Let $U_q := \left\langle A_\varepsilon \right\rangle$. By \cite[Theorem 2.1.(c)]{king2005subgroup},  there is a unique conjugacy class of cyclic groups of order $\frac{q+1}{2}$ in $\psl{2}{q}$. Hence, from \cite[Theorem 2.1(c), (g), (h)]{king2005subgroup}, $U_q$ is contained in a dihedral subgroup $V_q$ of order $q+1$. By \cite[Corollary 2.2(b)]{king2005subgroup}, when $q \neq 7$, the subgroup $V_q \leq \psl{2}{q}$ is maximal in $\psl{2}{q}$, and so $V_q$ is the normalizer of $U_q$ in $\psl{2}{q}$. It is not hard to verify that the latter also holds when $q=7$. 
				
			Consider the action of $\psl{2}{q}$ on cosets of $U_q$ by left multiplication. This action is transitive of degree $[\psl{2}{q}:U_q] = q(q-1)$. 
			For this action, an element of $\psl{2}{q}$ is a derangement if and only if it is not conjugate to an element of $U_q$, so its order does not divide $\frac{q+1}{2}$. Moreover, if $g\in \psl{2}{q}$ of order $d\mid \frac{q+1}{2}$, then by \cite[Theorem 2.1(c)]{king2005subgroup} there is a unique conjugacy class of cyclic subgroups of order $d$ in $\psl{2}{q}$, so $\langle g \rangle$ is conjugate to the subgroup of order $d$ of $U_q$. In other words, $g$ has a fixed point. Therefore, we obtain the following result.
			
			\begin{lem}\label{lem:charder}
				An element of $\psl{2}{q}$ in its action on cosets of $U_q$ is a derangement if and only if its order is not a divisor of $\frac{q+1}{2}$.
			\end{lem}
			
			We also observe that $V_q$ is an intersecting set of size twice the one of $U_q$.
			
			\begin{prop}\label{prop:vqinter}
				The subgroup $V_q$ is intersecting in the action of $\psl{2}{q}$ on cosets of $U_q$.
			\end{prop}
			\begin{proof}
				A subgroup is intersecting if and only if it is derangement-free, with respect to the action (see \cite{meagher180triangles}).  Let $g \in V_q$. Since $V_q$ is dihedral of order $q+1$ and $q \equiv 3 \pmod{4}$, the order of $g$ divides $\frac{q+1}{2}$. Hence, by Lemma \ref{lem:charder}, $g$ is not a derangement. This completes the proof. 						
			\end{proof}

			\begin{rmk}
				It is crucial to have $q \equiv 3 \pmod 4$ in Proposition \ref{prop:vqinter},  otherwise $2$ does not divide $\frac{q+1}{2}$, in which case $V_q$ is not intersecting.
			\end{rmk}
			
			Consequently, we immediately see that this action of $\psl{2}{q}$ does not have the EKR property. 
			
			\subsection{The proof of Theorem~\ref{thm:main}}
			 For the proof, it suffices to show that the largest intersecting set for the action of $\psl{2}{q}$  on cosets of $U_q$ is of size no more than $q+1$. Indeed, from Proposition \ref{prop:vqinter}, this bound is already attained by the intersecting subgroup $V_q$.

			 Now let $\Gamma_q$ be the derangement graph with respect to this action of $\psl{2}{q}$. Consider the following weighted adjacency matrix of $\Gamma_q$:
			\begin{align}
				\mathcal{A}_q = \frac{1}{q+1}\left(C_2^\prime(1)+C_2^\prime(-1)\right) + \frac{2}{q+1} \sum_{i = 1}^{\frac{q-3}{4}}C_3^\prime(\omega^i).\label{eq:adj1}
			\end{align}		
			It is worthwhile to note that  for a linear combination of the matrices $C_2^\prime(1),\ C_2^\prime(-1)$, and those in $\left\{C_3^\prime(\omega^i) \mid i \in \{1,2,\ldots,\frac{q-3}{4} \} \right\}$ to be symmetric, it is enough to impose that the weights on $C_2(1)$ and $C_2(-1)$ are equal, since $c_2(1)^{-1} = c_2(-1)$. Using Table~\ref{tab:q-cong2}, we deduce that the eigenvalues of \eqref{eq:adj1} are those in the next table.
			\begin{table}[H]
				\begin{tabular}{|c|c|c|c|c|c|}
					\hline
					&$\rho^\prime(1)$& $\overline{\rho}(1)$& $\rho(\alpha)$ & $\pi(\chi)$ & $\omega_0^\pm$\\
					\hline\hline
					Eigenvalue & $\frac{q(q-1)}{2}-1$ & $\frac{q-3}{2}$ & $-1$ & $-1$ & $\frac{-2}{q^2-1}$\\
					\hline
				\end{tabular}
				\caption{Eigenvalues of \eqref{eq:adj1}.}
			\end{table}
			Consequently, the Ratio bound (Theorem~\ref{lem:ratio-bound}) yields
			\begin{align*}
				\alpha(\Gamma_q) \leq \frac{|\psl{2}{q}|}{1-\frac{\frac{q(q-1)}{2}-1}{-1}} = q+1.
			\end{align*}
This completes the proof.

		\section{Proof of Theorem~\ref{thm2} for $q \equiv 1 \pmod 4$}\label{sect:thm-1mod4}

		Let $q \equiv 1 \pmod 4$.  For any $r \mid \frac{q-1}{2}$ odd, consider the action of $\psl{2}{q}$ on the subgroup $\mathbb{M}_r$  of Theorem \ref{thm:secon}. See Notations \ref{note:1e} below for an explicit description of $\mathbb{M}_r$.
		
		The proof given in this section is quite lengthy, so first we would like to give the steps of the proof. The main tool used in the proof is again Theorem~\ref{lem:ratio-bound}. To find a weighted adjacency matrix for the derangement graph of the action of $\psl{2}{q}$ on cosets of $\mathbb{M}_r$, one needs to identify the derangements. In Subsection~\ref{subs1}, we determine the derangements and the number of cosets fixed by non-derangements for the action in question. Then, we find the permutation character in Subsection~\ref{subs2} and the character values on conjugacy classes of derangements in Subsection~\ref{subs3}. Finally, in Subsection~\ref{subs4} we give a weighted adjacency matrix for which the upper bound in Theorem~\ref{lem:ratio-bound} yields the desired value.
		
			\subsection{Number of fixed points}\label{subs1}
			We will determine the number of fixed points for each conjugacy classes of $\psl{2}{q}$, in its action on $\mathbb{M}_r$.
			\begin{note}\label{note:1e}
					 Throughout this paper, besides the notations from \S\ref{sec:conjq14}, we will use the following additional ones.
					\begin{enumerate}[(N1)]
						\item Let
						$\mathcal{I}_q=\left\lbrace 1,\dots, \frac{q-5}{4}\right \rbrace$. Note that $\{ c_3(\omega^i) \mid i \in \mathcal{I}_q \}$ is a complete set of representatives of conjugacy classes of  type (T\ref{type4}). Also, we have $|\mathcal{I}_q| = \frac{q-5}{4}$.
						\item  Let $\mathcal{Z}_q \subset E_q \setminus \{\pm 1\}$ such that $\{c_4(z):\ z\in \mathcal{Z}_q \}$ is a complete set of representatives of conjugacy classes of type (T\ref{type6}). Note that $|\mathcal{Z}_q| = \frac{q-1}{4}$.
						\item Recall that the Borel subgroup $\mathbb{B}_q$ consists of all upper triangular matrices. Consider the subgroups of $\psl{2}{q}$:
						\begin{align*}
							H = \left\{ 
							\overline{				
								\begin{bmatrix}
									1 & x\\
									0 & 1
							\end{bmatrix}} \ :\ 
							x \in\mathbb{F}_q
							\right\} \mbox{ and }
							K = \left\langle A  \right\rangle,
						\end{align*}
						where $A = \overline{\begin{bmatrix}
								\omega & 0\\
								0 & \omega^{-1}
						\end{bmatrix}}$. Noticing that $A$ normalizes $H$, $|\langle A \rangle|=\frac{q-1}{2}$ and $H \cap A =\{\overline{I}\}$, we obtain  $\mathbb{B}_q=H \rtimes \langle A \rangle \simeq \mathbb{F}_q \rtimes \mathbb{Z}_{\frac{q-1}{2}}$.  For any odd $r \mid \frac{q-1}{2}$, we have $\mathbb{M}_r := \langle H,A^{r} \rangle = H \rtimes \langle A^{r} \rangle \simeq \mathbb{F}_q \rtimes \mathbb{Z}_{\frac{q-1}{2r}} $. It is easy to check that $[\mathbb{B}_q:\mathbb{M}_r]=r$.
					\end{enumerate}\label{not1}
			\end{note}
				
				\begin{lem}
					For $x \in \{1,\Delta\}$, the matrix $c_2(x)$ admits $r$ fixed cosets.\label{lem1}
				\end{lem}
				\begin{proof}
					Let $x \in \{1,\Delta\}$.
					For any $g=\overline{\begin{bmatrix}
							a & b\\
							c & d
					\end{bmatrix}} \in \psl{2}{q}$, we have
					$c_2(x)\cdot g\mathbb{M}_r=g\mathbb{M}_r$, if and only if,
					\begin{align*} g^{-1}c_{2}(x) g=
						\overline{\begin{bmatrix}
								a & b\\
								c & d
						\end{bmatrix}}^{-1}
						\overline{\begin{bmatrix}
								1 & x\\
								0 & 1
						\end{bmatrix}}
						\overline{\begin{bmatrix}
								a & b\\
								c&d
						\end{bmatrix}}
						= 
						\overline{\begin{bmatrix}
								1+cd x & d^2 x\\
								-c^2 x & 1 - cd x
						\end{bmatrix}} \in \mathbb{M}_r,
					\end{align*}
					if and only if, $c=0$; or in other words, $g\mathbb{M}_r \subset \mathbb{B}_q$.
					Since $[\mathbb{B}_q:\mathbb{M}_r]=r$, we deduce that $c_2(x)$ has exactly $r$ fixed cosets, contained in $\mathbb{B}_q$,  which completes the proof.
				\end{proof}
				\begin{lem}\label{lem:c3fixss}
					Let $ i\in \left\{1,2,\ldots,\frac{q-1}{4} \right\}=\mathcal{I}_q \cup \{ \frac{q-1}{4}\}$. If $i \equiv 0\pmod{r}$, then $c_3(\omega^i)$ fixes exactly $2r$ cosets. Moreover, if  $i \not\equiv 0 \pmod {r}$, then $c_3(\omega^i)$ is a derangement.
				\end{lem}
				\begin{proof}
					For any $g=\overline{\begin{bmatrix}
							a & b \\
							c & d
					\end{bmatrix}} \in \psl{2}{q}$, we have $c_3(\omega^i) \cdot g \mathbb{M}_r=g\mathbb{M}_r$, if and only if,
					\begin{align}\label{eq:belongMr}
						g^{-1}c_{3}(\omega^i)g=
						\overline{\begin{bmatrix}
								a & b\\
								c & d
						\end{bmatrix}}^{-1}
						\overline{\begin{bmatrix}
								\omega^i & 0\\
								0 & \omega^{-i}
						\end{bmatrix}} 
						\overline{\begin{bmatrix}
								a & b \\
								c & d
						\end{bmatrix}}
						= 
						\overline{\begin{bmatrix}
								ad\omega^i - bc \omega^{-i} & bd ( \omega^i - \omega^{-i})\\
								ac\left( \omega^{-i} - \omega^i\right) & ad \omega^{-i} - bc \omega^i
						\end{bmatrix}} \in \mathbb{M}_r,
					\end{align}
					if and only if, $ac (\omega^{-i}-\omega^i) = 0$, that is, $ac=0$.
					
					Assume $a = 0$.  The assertion \eqref{eq:belongMr}, which becomes
					\begin{align*}
						g^{-1}c_{3}(\omega^i)g
						= 
						\overline{\begin{bmatrix}
								- bc \omega^{-i} & bd ( \omega^i - \omega^{-i})\\
								0 & - bc \omega^i
						\end{bmatrix}}
						= 
						\overline{\begin{bmatrix}
								-\omega^{-i} & -bd (\omega^i - \omega^{-i})\\
								0 & -\omega^i	
						\end{bmatrix}} \in \mathbb{M}_r,
					\end{align*}
					holds,
					if and only if, $\pm \omega^i \in \langle \omega^r\rangle$, if and only if, {$i \equiv 0 \pmod{r}$ or $i+\frac{q-1}{2} \equiv 0 \pmod{r}$, if and only if, $i \equiv 0 \pmod{r}$ (since $r \mid \frac{q-1}{2}$}).
					Similarly, when $c = 0$, then \eqref{eq:belongMr} is also equivalent to $i \equiv 0\pmod{r}$. This proves the second part of the theorem.
					
					Next, assume that $i \equiv 0 \pmod r$. Let us prove that $c_3(\omega^i)$ fixes exactly $2r$ cosets of $\mathbb{M}_r$. As we saw in the above paragraph,  $c_3(\omega^i)$ fixes a coset $g\mathbb{M}_r$ if and only if $g$ is a matrix of the form
					\begin{align}
						\overline{\begin{bmatrix}
								a & b\\
								0 & a^{-1}
						\end{bmatrix}}
						\mbox{ or }
						\overline{\begin{bmatrix}
								0 & b\\
								-b^{-1} & d
						\end{bmatrix}}.
					\label{eq:matrices}
					\end{align}
										
					Cosets $g\mathbb{M}_r$ of the first type are  exactly all cosets of $\mathbb{M}_r$ in $\mathbb{B}_q$. Hence, there are $r$ of  them. Now let us count cosets $g\mathbb{M}_r$ of the second type, which are clearly not of the first type. For $b,\beta\in \mathbb{F}_q^*$ and $d,\gamma\in \mathbb{F}_q$, we have
					\begin{align*}
						\overline{\begin{bmatrix}
								0 & b\\
								-b^{-1} & d
						\end{bmatrix}}^{-1}
						\overline{\begin{bmatrix}
								0& \beta\\
								-\beta^{-1} & \gamma 
						\end{bmatrix}}
						=
						\overline{\begin{bmatrix}
								b\beta^{-1}  & d\beta-b\gamma\\
								0 & b^{-1}\beta
						\end{bmatrix}} \in \mathbb{M}_r \Leftrightarrow b^{-1}\beta \in \langle \omega^{r} \rangle.
					\end{align*}
					Consequently, there are exactly $r$ cosets of the second type. We conclude that $c_3(\omega^i)$ fixes exactly $2r$ cosets. This completes the proof.
				\end{proof}
				\begin{lem}\label{lem:c4isder}
					For all  $z \in \operatorname{E}_q \setminus \{\pm 1\}$, the matrix $c_4(z)$ is a derangement.
				\end{lem}
				\begin{proof}
					Let $z=a+b \delta \in \operatorname{E}_q \setminus \{\pm 1\}$ with $a, b \in \mathbb{F}_q$. Note that the characteristic polynomial of any lift of the matrix $c_4(z)$ in ${\rm SL}_2(q)$, has discriminant $\Delta b^2$, which is not a square in $\mathbb{F}_q$; and so it is irreducible over $\mathbb{F}_q$.  Now let $g\in \psl{2}{q}$. We claim that $g^{-1}c_4(z) g$ does not belong to $\mathbb{M}_r$; in other words  $c_4(z)\cdot g \mathbb{M}_r \neq g \mathbb{M}_r$. Indeed, matrices in $\mathbb{M}_r$ are upper triangular, and so their lifts in ${\rm SL}_2(q)$ have reducible characteristic polynomials over $\mathbb{F}_q$. This completes the proof. 
			\end{proof}
			\subsection{The permutation character}\label{subs2}
			
			We consider again the notations for irreducible characters of $\psl{2}{q}$ from \S\ref{sec:charpsl22}. Let us determine the permutation character of $\psl{2}{q}$ in its action on cosets of $\mathbb{M}_r$. Let $\xi$ be a primitive $(q-1)$-root of unity. For any $i\in \{1,\ldots,q-1\}$, define $\alpha_i: \mathbb{F}_q^* \to \mathbb{C}$ such that $\alpha_i(\omega^j) = \xi^{ij}$, for all $j \pmod{q-1}$. Clearly, we have $\operatorname{Irr}(\mathbb{F}_q^*) = \left\{ \alpha_i : i \in \{1,\ldots,q-1\} \right\}$. 
			
			It is worth noting that, for any {$i\in \{1,\ldots, r-1\}$}, the restriction of $\alpha_{\frac{q-1}{r}i}$ to the subgroup $\langle \omega^r \rangle$ coincides with the trivial character of the latter. Moreover, for any distinct $i, j \in \{1,\ldots,r-1 \}$, we have $$\rho \left(\alpha_{\frac{q-1}{r}i}\right) = \rho\left(\alpha_{\frac{q-1}{r}j}\right)$$ if and only if $j\equiv -i \pmod{r}$. Consequently, there are $\frac{r-1}{2}$ distinct non-trivial characters of the form $\rho(\alpha_j) \in \operatorname{Irr}(\psl{2}{q})$ ($j \equiv 0\pmod {\tfrac{q-1}{r}}$).
			
			In the next lemma, we determine the decomposition of the permutation character of the action of $\psl{2}{q}$ on cosets of $\mathbb{M}_r.$ This result is not entirely essential to the proof of Theorem~\ref{thm2}; it rather tells us from which irreducible characters to expect the smallest eigenvalue of a suitable weighted adjacency matrix to come from. In fact, this will be the case as we will see in Corollary~\ref{cor}. Nevertheless, we give the proof of this result here for completeness. 
			\begin{lem}
				If $\Psi$ is the permutation character of $\psl{2}{q}$, in its action on $\mathbb{M}_r$, then
				\begin{align}
					\Psi = \rho^{\prime}(1) + \overline{\rho}(1) +  2\sum_{j = 1}^{\tfrac{r-1}{2}} \rho\left(\alpha_{\frac{q-1}{r}j}\right).\label{eq:perm-char}
				\end{align}
			\end{lem}
			\begin{proof}
				Note that  
				\begin{align*}
					\Psi\colon \psl{2}{q} &\rightarrow \mathbb{C}\\
					B& \mapsto\Psi(B)=\mathbf{fix}_{\mathbb{M}_r}^{\psl{2}{q}}(B) = \left|\left\{ x\mathbb{M}_r \in \psl{2}{q}/\mathbb{M}_r : x^{-1}Bx \in \mathbb{M}_r \right\}\right|.
				\end{align*}
				Since characters are constant on conjugacy classes, it is enough to show that 
				$$\Psi = \rho^{\prime}(1) + \overline{\rho}(1) +  2\sum_{j = 1}^{\tfrac{r-1}{2}} \rho\left(\alpha_{\frac{q-1}{r}j}\right),$$ on representatives of the conjugacy classes. 
				\begin{itemize}
					\item We see that $\Psi(\overline{I}) =|\psl{2}{q}/\mathbb{M}_r|= r(q+1)$. Moreover, we have
					\begin{align*}
						\rho^\prime (1)(\overline{I}) + \overline{\rho}(1)(\overline{I}) + 2\sum_{j = 1}^{\frac{r-1}{2}} \rho\left( \alpha_{\frac{q-1}{r}j} \right) (\overline{I}) = 1+q + (r-1)(q+1) = r(q+1).
					\end{align*}
					\item For the conjugacy class containing $c_2(1)$, we have $\Psi(c_2(1)) = r$, by Lemma \ref{lem1}. Moreover, we have
					\begin{align*}
						\rho^\prime(1) (c_2(1)) + \overline{\rho}(1) (c_2(1)) + 2 \sum_{j = 1}^{\frac{r-1}{2}} \rho \left(\alpha_{\frac{q-1}{r}j}\right) (c_2(1)) &= 1 + 0 +2\times \tfrac{r-1}{2} = r.
					\end{align*}
					Similarly, the equality holds for $c_2(\Delta)$.
					\item For the conjugacy class containing $c_3(\omega^i)$ ($1 \leq i \leq \tfrac{q-5}{4}$), by Lemma \ref{lem:c3fixss}, we have
					\begin{align*}
						\Psi(c_3(\omega^i)) =
						\begin{cases}
							2r & \mbox{ if }i \equiv 0 \pmod {r}, \\
							0 & \mbox{ otherwise.}
						\end{cases} 
					\end{align*}
					If $i \equiv 0 \pmod{r}$, then 
					\begin{align*}
						\rho^{\prime}(1) (c_3(\omega^i)) + \overline{\rho}(1) (c_3(\omega^i)) +  2\sum_{j = 1}^{\tfrac{r-1}{2}} \rho\left(\alpha_{\frac{q-1}{r}j}\right) (c_3(\omega^i)) &=
						1 + 1 + 2 \times \frac{r-1}{2} \times 2 = 2r.
					\end{align*}
					 If $i \not \equiv 0 \pmod{r}$, then 
					\begin{align*}
						\rho^{\prime}(1) (c_3(\omega^i)) + \overline{\rho}(1) (c_3(\omega^i)) +  2\sum_{j = 1}^{\tfrac{r-1}{2}} \rho\left(\alpha_{\frac{q-1}{r}j}\right) (c_3(\omega^i)) &=
						1 + 1 + 2 \times \sum_{j=1}^{r-1}\left(\xi^{\frac{q-1}{r}i}\right)^{j}  = 2-2 = 0.
					\end{align*}
					\item For the conjugacy class of $c_3(\sqrt{-1})$, we have $\Psi(c_3(\sqrt{-1})) = 2r$, by Lemma \ref{lem:c3fixss}. Moreover, we have 
					\renewcommand\thefootnote{\fnsymbol{footnote}}
					\begin{align*}
						\rho^{\prime}(1) (c_3(\sqrt{-1})) + \overline{\rho}(1) (c_3(\sqrt{-1})) +  2\sum_{j = 1}^{\tfrac{r-1}{2}} \rho\left(\alpha_{\frac{q-1}{r}j}\right) (c_3(\sqrt{-1})) &= 1+1+ 2 \sum_{j=1}^{\frac{r-1}{2}} 2 \xi^{\frac{(q-1)^2}{4r}j} \\&=  1 + 1 + 2 \times \tfrac{r-1}{2} \times 2 \footnotemark\\
						&= 2r.
					\end{align*}
					\footnotetext{as $4r\mid(q-1)$}
					\item For the conjugacy class $c_4(z)$, where  $z \in \operatorname{E}_q\setminus\{\pm 1\}$, we have $\Psi(c_4(z)) = 0$, by Lemma \ref{lem:c4isder}. Moreover, 
					\begin{align*}
						\rho^{\prime}(1)(c_4(z)) + \overline{\rho}(1)(c_4(z)) +  2\sum_{j = 1}^{\tfrac{r-1}{2}} \rho\left(\alpha_{\frac{q-1}{r}j}\right)(c_4(z)) &= 1-1 = 0.
					\end{align*}
				\end{itemize}
				Hence, we conclude that $\Psi = \rho^{\prime}(1) + \overline{\rho}(1) +  2\sum_{j = 1}^{\tfrac{r-1}{2}} \rho\left(\alpha_{\frac{q-1}{r}j}\right)$.
			\end{proof}
			
			\subsection{Character values on derangements}\label{subs3}
			Next, we compute some character values of $\psl{2}{q}$ that are needed in the proof of the main result. We first note that there are
			
				\begin{align*}
					\frac{q-5}{4}- \left\lfloor \frac{q-5}{4r} \right\rfloor=\frac{q-1}{4} - \frac{q-1}{4r} = \frac{(r-1)(q-1)}{4r},
				\end{align*}
elements in $\mathcal{I}_q$ (see Notations~\ref{not1}) that are not multiple of $r$. 
			\begin{note}
				Let $\mathcal{J}_q = \left\{ i \in \mathcal{I}_q : r \nmid i \right\}$.\label{note2}
			\end{note}
			
			\begin{lem}
				If $\alpha \in \operatorname{Irr}(\mathbb{F}_q^*)$ is a non-trivial character {with $\alpha(-1)=1$}, then
				\begin{align*}
					\sum_{i \in \mathcal{J}_q} \left(\alpha(\omega^i) + \alpha(\omega^{-i})\right)
					=
					\begin{cases}
						-\frac{q-1}{2r} & \mbox{ if } \alpha_{|\langle \omega^r \rangle} \mbox{ is the trivial character of $\langle \omega^r \rangle$,} \\
						0 & \mbox{ otherwise.}
					\end{cases}
				\end{align*}\label{lem:first}
			\end{lem}
			\begin{proof}
				Since $\alpha$ is non-trivial, by orthogonality of characters, we have
				\begin{align*}
					0 = \sum_{i=1}^{q-1} \alpha(\omega^i) &= \sum_{r \mid i} \alpha(\omega^{i}) + \sum_{r \nmid i} \alpha(\omega^i)
					\\
					&=\sum_{i = 1}^{\frac{q-1}{r}} \alpha(\omega^{ri}) + \sum_{i \in \mathcal{J}_q} \left(\alpha(\omega^i)+ \alpha(\omega^{-i})\right) + \sum_{i \in \mathcal{J}_q} \left(\alpha(-\omega^i)+ \alpha(-\omega^{-i})\right)\\
					&= \sum_{i = 1}^{\frac{q-1}{r}} \alpha(\omega^{ri}) + 2 \sum_{i \in \mathcal{J}_q} \left(\alpha(\omega^i)+ \alpha(\omega^{-i})\right) .
				\end{align*}
				Therefore, we get
				\begin{align*}
					\sum_{i\in \mathcal{J}_q} (\alpha(\omega^i) + \alpha(\omega^{-i})) &= - \frac{1}{2} \sum_{i=1}^{\frac{q-1}{r}}\alpha(\omega^{ri}) = -\frac{1}{2} \sum_{x \in \langle \omega^r\rangle} \alpha(x).
				\end{align*}
				The right-hand-side of the above equation depends on the restriction of $\alpha$ onto the subgroup $\langle \omega^r\rangle$. If this restriction is non-trivial, then the right-hand-side is equal to $0$ by orthogonality of characters, otherwise it is equal to $-\frac{q-1}{2r}$. This completes the proof. 
			\end{proof}
		
			\begin{lem}
				If $\chi \in \operatorname{Irr}(E_q)$ is a non-trivial character  such that $\chi(-1)=1$, then we have
				\begin{align*}
					\sum_{z \in \mathcal{Z}_q} (\chi(z) + \chi(z^{-1})) = -1.
				\end{align*}\label{lem:second}
			\end{lem}

			\begin{proof}
			Let $E_q=\langle \omega' \rangle$. Note that $\omega'$ has order $q+1$. Without loss of generality, we can assume $\mathcal{Z}_q=\left\lbrace \omega'^i \mid 1 \leq i \leq \tfrac{q-1}{4}\right\rbrace$. Since
			$$
			0=\sum_{z \in 	E_q}\chi(z)=\chi(\omega'^0)+\chi\left(\omega'^{\tfrac{q+1}{2}}\right)+\sum_{i=1}^{\tfrac{q-1}{4}}\left(\chi(\omega'^i)+\chi(\omega'^{-i})\right)+\chi\left(\omega'^{\tfrac{q+1}{2}}\right)\sum_{i=1}^{\tfrac{q-1}{4}}\left(\chi(\omega'^i)+\chi(\omega'^{-i})\right),
			$$
			we get
			$$
			\sum_{z \in 	\mathcal{Z}_q}(\chi(z)+\chi(z^{-1}))=-\frac{\chi(\omega'^0)+\chi\left(\omega'^{\tfrac{q+1}{2}}\right)}{1+\chi\left(\omega'^{\tfrac{q+1}{2}}\right)}=-\frac{\chi(1)+\chi(-1)}{1+\chi(-1)}=-1,
			$$
			which completes the proof.
			\end{proof}

			\begin{lem}
				Let $\mathcal{J}_q$ be as in Notations~\ref{note2}. Then, we have 
				\begin{align*}
					\sum_{i \in \mathcal{J}_q} \zeta(\omega^i) = 
					0,
				\end{align*} \label{lem:third}
where $\zeta \in \operatorname{Irr}(\mathbb{F}_q^*)$ denotes the unique irreducible character of $\mathbb{F}_q^*$ such that $\zeta(\omega) = -1$.			
			\end{lem}	
				\begin{proof}
					We have
					$$
					\sum_{i=1}^{\frac{q-1}{4}}\zeta(\omega^i)=\sum_{i \in \mathcal{J}_q}\zeta(\omega^i)+\sum_{i=1,r|i}^{\frac{q-1}{4}}\zeta(\omega^i)=\sum_{i \in \mathcal{J}_q}\zeta(\omega^i)+\sum_{i=1}^{\frac{q-1}{4r}}\zeta((\omega^r)^i),
					$$
					and so
					$$
					\sum_{i \in \mathcal{J}_q}\zeta(\omega^i)=\sum_{i=1}^{\frac{q-1}{4}}\zeta(\omega^i)-\sum_{i=1}^{\frac{q-1}{4r}}\zeta((\omega^r)^i)=\sum_{i=1}^{\frac{q-1}{4}}(-1)^i-\sum_{i=1}^{\frac{q-1}{4r}}((-1)^r)^i=\sum_{i=1}^{\frac{q-1}{4}}(-1)^i-\sum_{i=1}^{\frac{q-1}{4r}}(-1)^i=0.
					$$
					For the last equality, we use the fact that $(-1)^r=-1$ (since $r$ is odd); and the fact that $\frac{q-1}{4}$ and $\frac{q-1}{4r}$ have the same parity, combined with
					$$
					\sum_{i=1}^j(-1)^i=\begin{cases}
						0 & \textrm{if}\,\, j \equiv 0 \pmod{2},\\
						-1 & \textrm{otherwise}.
					\end{cases}
					$$
				\end{proof}
			
			\subsection{Proof of Theorem \ref{thm2}}\label{subs4}
			Let $\Gamma_q$ be the derangement graph of the action of $\psl{2}{q}$ on cosets of $\mathbb{M}_r$. Consider the following weighted adjacency matrix of $\Gamma_q$
			\begin{align}\label{eq:B}
				\displaystyle B_q = \left(\frac{(q+1)r - (q+1)}{2 \left(q(q^2-1)(\frac{r-1}{4r})\right)}\right) \sum_{i \in \mathcal{J}_q} C_3(i) + \left( \frac{(q+1)r + (q-1)}{2\left(\frac{q(q-1)^2}{4}\right)} \right) \sum_{z\in \mathcal{Z}_q}C_4(z).
			\end{align}
			
			Using the formula in Theorem~\ref{thm:evalues} for eigenvalues, Lemma~\ref{lem:first}, Lemma~\ref{lem:second} and Lemma~\ref{lem:third} for the character values, we deduce that the  eigenvalues of $B_q$ are given in the following table.
			\begin{table}[H]
				\begin{tabular}{|c|c|c|c|c|c|c|}
					\hline
					&$\rho^\prime(1)$& $\overline{\rho}(1)$& $\rho(\alpha)\,\, (\alpha_{|\langle \omega^r \rangle}=1)$ & $\rho(\alpha)\,\,(\alpha_{|\langle \omega^r \rangle}\neq 1)$ & $\pi(\chi)$ & $\omega_e^\pm$\\
					\hline\hline
					Eigenvalue & $r(q+1)-1$  & $-1$  &  $-1$&$0$  & $2 \left(\frac{r+1}{q-1}+ \frac{2r}{(q-1)^2}\right)$&$0$\\
					\hline
				\end{tabular}\caption{Eigenvalues of the matrix $B_q$} \label{lem:eigenvalues-adj}
			\end{table}				
			
			Using Theorem~\ref{lem:ratio-bound} and the values in Table~\ref{lem:eigenvalues-adj}, we conclude that
			\begin{align*}
				\alpha(\Gamma_q) \leq \frac{|\psl{2}{q}|}{1- \frac{r(q+1)-1}{-1}} = \frac{|\psl{2}{q}|}{r(q+1)},
			\end{align*} 
			which coincides with the order of a stabilizer of a point. In other words, we get $\rho(\psl{2}{q},\mathbb{M}_r) = 1$. This completes the proof of Theorem \ref{thm2}.
					
		In addition, as mentioned in the introduction, we can also deduce the following slightly stronger result.
					
		\begin{cor}\label{rem:pluthm6}
		If $S \subset \psl{2}{q}$ is a coclique of maximum size of $\Gamma_q$, then 
					\begin{align*}
						v_S \in V_{\rho^\prime(1)} \oplus V_{\overline{\rho}(1)} 	\oplus  \bigoplus_{i=1}^{\frac{r-1}{2}} V_{\rho\left(\alpha_{\frac{q-1}{r}i}\right)}.
					\end{align*}\label{cor}
		\end{cor}		
		\begin{proof}
		Using Table~\ref{lem:eigenvalues-adj} the smallest eigenvalue of $B$ is $-1$ and the irreducible characters affording this eigenvalue are constituents of the permutation character of $\psl{2}{q}$ acting on $\mathbb{M}_r$. From this, we deduce using Theorem~\ref{thm:evalues} that the eigenspace of $B_q$ corresponding to the eigenvalue $-1$ is
		\begin{align*}
			 V_{\overline{\rho}(1)} \oplus  \bigoplus_{i=1}^{\frac{r-1}{2}} 	V_{\rho\left(\alpha_{\frac{q-1}{r}i}\right)}.
		\end{align*}  
		It  remains then to apply Theorem~\ref{lem:ratio-bound}.
		\end{proof}	
			
			\section{Concluding remark}
			
			The main purpose of this paper was to find rational numbers larger than $1$ in the intersection spectrum of $\psl{2}{q}$, for $q$ an odd prime power. The latter is equivalent to finding actions of $\psl{2}{q}$ that do not have the EKR property. We showed that $2 \in \sigma(\psl{2}{q})$, for $q\equiv 3 \pmod 4$. For the case $q\equiv 1 \pmod 4$, the group $\mathbb{Z}_{\frac{q-1}{2}}$ is a subgroup of $\psl{2}{q}$ and the action of $\psl{2}{q}$ on this subgroup yields an intersecting set of size $q-1$, which is given by the normalizer of $\mathbb{Z}_{\frac{q-1}{2}}$.  We were unfortunately unable to determine $\rho\left(\psl{2}{q},\mathbb{Z}_{\frac{q-1}{2}}\right)$ with the current tools at our disposal. Nevertheless, based on computational results (see Appendix~\ref{app}), we conjecture the following.
			\begin{conj}
				 For any $q\equiv 1 \pmod 4$, we have $\rho \left(\psl{2}{q},\mathbb{Z}_{\frac{q-1}{2}}\right) = 2$.
			\end{conj}
			
			In this paper, we also gave an example of an action of $\psl{2}{q}$ that has the EKR property. To classify $\sigma \left( \psl{2}{q}\right)$, one would need to determine all actions that have the EKR property. Hence, the next problem is another possible direction for future research.
			\begin{prob}
				 Classify all subgroups $H\leq \psl{2}{q}$ for which $\rho(\psl{2}{q},H) = 1$.
			\end{prob}
			In Theorem~\ref{thm2}, we found the family of subgroups $\mathbb{M}_r$ of $\psl{2}{q}$ for which we have $\rho \left(\psl{2}{q},\mathbb{M}_r\right) = 1$. When $r=1$, this action of $\psl{2}{q}$ is on the projective line $\pg{1}{q}$, and it has been proved to have the EKR and the strict-EKR property (see \cite{long2018characterization,meagher2011erdHos}). As proved in Corollary~\ref{cor}, we also obtained the irreducible modules that contain the characteristic vectors of maximum intersecting sets of this action. This fact is an intermediate step in proving the strict-EKR property. Unfortunately, we could not determine whether this group action has the strict-EKR property or no. We leave this as an open question.
			\begin{qst}
				For $q\equiv 1 \pmod 4$, does the action of $\psl{2}{q}$ on cosets of $\mathbb{M}_r$ have the strict-EKR property?
			\end{qst}

\section*{Acknowledgements}
The first author is grateful for the support of the Israel Science Foundation (grant no. 353/21). The second and the third authors are supported in part by the Ministry of
Education, Science and Sport of Republic of Slovenia (University of Primorska Developmental funding pillar).
			
			%%%%%%%%%%%%%%%%%%%%%
			
%			\bibliography{ref-join}

\begin{thebibliography}{10}
				
				\bibitem{adams2002character}
				J.~Adams.
				\newblock Character tables for {GL} (2),{ SL} (2), {PGL} (2) and {PSL} (2) over
				a finite field.
				\newblock {\em Lecture Notes, University of Maryland}, {\bf 25}:26--28, 2002.
				
				\bibitem{babai1979spectra}
				L.~Babai.
				\newblock {S}pectra of {C}ayley graphs.
				\newblock {\em J. Combin. Ser. B}, {\bf 27}(2):180--189, 1979.
				
				\bibitem{bardestani2015erdHos}
				M.~Bardestani and K.~Mallahi-Karai.
				\newblock On the {E}rd{\H{o}}s-{K}o-{R}ado property for finite groups.
				\newblock {\em J. Algebraic Combin.}, 42(1):111--128, 2015.
				
				\bibitem{cameron2003intersecting}
				P.~J. Cameron and C.~Y. Ku.
				\newblock Intersecting families of permutations.
				\newblock {\em European J. Combin.}, {\bf 24}(7):881--890, 2003.
				
				\bibitem{Frankl1977maximum}
				M.~Deza and P.~Frankl.
				\newblock On the maximum number of permutations with given maximal or minimal
				distance.
				\newblock {\em J. Combin. Ser. A}, {\bf 22}(3):352--360, 1977.
				
				\bibitem{erdos1961intersection}
				P.~Erd\H{o}s, C.~Ko, and R.~Rado.
				\newblock Intersection theorems for systems of finite sets.
				\newblock {\em Q. J. Math.}, {\bf 12}(1):313--320, 1961.
				
				\bibitem{fulton2013representation}
				W.~Fulton and J.~Harris.
				\newblock {\em Representation theory: a first course}, volume~{\bf 129}.
				\newblock Springer Science \& Business Media, 2013.
				
				\bibitem{godsil2009new}
				C.~Godsil and K.~Meagher.
				\newblock A new proof of the {E}rd{\H{o}}s--{K}o--{R}ado theorem for
				intersecting families of permutations.
				\newblock {\em European J. Combin.}, {\bf 30}(2):404--414, 2009.
				
				\bibitem{godsil2016erdos}
				C.~Godsil and K.~Meagher.
				\newblock {\em {E}rd\H{o}s-{K}o-{R}ado {T}heorems: {A}lgebraic {A}pproaches}.
				\newblock Cambridge University Press, 2016.
				
				\bibitem{hujdurovic2022intersection-cyclic}
				A.~Hujdurovi{\'c}, I.~Kov{\'a}cs, K.~Kutnar, and D.~Maru{\v{s}}i{\v{c}}.
				\newblock Intersection density of transitive groups with cyclic point
				stabilizers.
				\newblock {\em arXiv preprint arXiv:2201.11015}, 2022.
				
				\bibitem{king2005subgroup}
				O.~H. King.
				\newblock The subgroup structure of finite classical groups in terms of
				geometric configurations.
				\newblock In {\em BCC}, pages 29--56, 2005.
				
				\bibitem{kutnar2023intersection}
				K.~Kutnar, D.~Maru{\v{s}}i{\v{c}}, and C.~Pujol.
				\newblock Intersection density of cubic symmetric graphs.
				\newblock {\em J. Algebraic Combin.}, pages 1--14, 2023.
				
				\bibitem{larose2004stable}
				B.~Larose and C.~Malvenuto.
				\newblock {S}table sets of maximal size in {K}neser-type graphs.
				\newblock {\em European J. Combin.}, {\bf 25}(5):657--673, 2004.
				
				\bibitem{long2018characterization}
				L.~Long, R.~Plaza, P.~Sin, and Q.~Xiang.
				\newblock Characterization of intersecting families of maximum size in
				${PSL}(2, q)$.
				\newblock {\em J. Combin. Ser A}, {\bf 157}:461--499, 2018.
				
				\bibitem{meagher2022intersection}
				K.~Meagher and A.~S. Razafimahatratra.
				\newblock On the intersection density of the {K}neser {G}raph $ {K}(n, 3) $.
				\newblock {\em arXiv preprint arXiv:2205.05118}, 2022.
				
				\bibitem{meagher180triangles}
				K.~Meagher, A.~S. Razafimahatratra, and P.~Spiga.
				\newblock On triangles in derangement graphs.
				\newblock {\em J. Combin. Ser. A}, {\bf 180}:105390, 2021.
				
				\bibitem{meagher2011erdHos}
				K.~Meagher and P.~Spiga.
				\newblock An {E}rd{\H{o}}s--{K}o--{R}ado theorem for the derangement graph of
				${PGL}(2, q)$ acting on the projective line.
				\newblock {\em J. Combin. Ser. A}, {\bf 118}(2):532--544, 2011.
				
				\bibitem{meagher2016erdHos}
				K.~Meagher, P.~Spiga, and P.~H. Tiep.
				\newblock An {E}rd{\H{o}}s--{K}o--{R}ado theorem for finite 2-transitive
				groups.
				\newblock {\em European J. Combin.}, {\bf 55}:100--118, 2016.
				
				\bibitem{sagemath}
				{The Sage Developers}.
				\newblock {\em {S}ageMath, the {S}age {M}athematics {S}oftware {S}ystem
					({V}ersion 10.0)}, 2023.
				\newblock {\tt https://www.sagemath.org}.
				
			\end{thebibliography}
%			\bibliographystyle{abbrv}

			\appendix
			
			\section{Computations}\label{app}
			In this section, we give some computations that were done on \verb*|Sagemath| \cite{sagemath}.
			
			\twocolumn
			\begin{table}[H]
			\begin{tabular}{|c|c|}
				\hline
				Subgroups $H$ & $\rho(\psl{2}{3},H)$ \\
				\hline\hline
				\verb|1| & $1$ \\
				\verb|C2| & $2$ \\
				\verb|C3| & $1$ \\
				\verb|C2 x C2| & $1$ \\
				\verb|A4| & $1$ \\ \hline
			\end{tabular}
			\end{table} 

			\begin{table}[H]
				\begin{tabular}{|c|c|}
				\hline
				Subgroups $H$ & $\rho(\psl{2}{4},H)$ \\ \hline\hline
				\verb|1| & $1$ \\
				\verb|C2| & $2$ \\
				\verb|C3| & $\frac{4}{3}$ \\
				\verb|C2 x C2| & $1$ \\
				\verb|C5| & $1$ \\
				\verb|S3| & $2$ \\
				\verb|D5| & $1$ \\
				\verb|A4| & $1$ \\
				\verb|A5| & $1$ \\ \hline
			\end{tabular} 
			\end{table}
			
			\begin{table}[H]
				\begin{tabular}{|c|c|}
				\hline
				Subgroups $H$ & $\rho(\psl{2}{5},H)$ \\
				\hline \hline
				\verb|1| & $1$ \\
				\verb|C2| & $2$ \\
				\verb|C3| & $\frac{4}{3}$ \\
				\verb|C2 x C2| & $1$ \\
				\verb|C5| & $1$ \\
				\verb|S3| & $2$ \\
				\verb|D5| & $1$ \\
				\verb|A4| & $1$ \\
				\verb|A5| & $1$ \\
				\hline
			\end{tabular} 
			\end{table}
			
			\begin{table}[H]
				\begin{tabular}{|c|c|}
				\hline
				Subgroups $H$ & $\rho(\psl{2}{7},H)$ \\ \hline \hline
				\verb|1| & $1$ \\
				\verb|C2| & $2$ \\
				\verb|C3| & $\frac{4}{3}$ \\
				\verb|C2 x C2| & $1$ \\
				\verb|C2 x C2| & $1$ \\
				\verb|C4| & $2$ \\
				\verb|S3| & $2$ \\
				\verb|C7| & $1$ \\
				\verb|D4| & $1$ \\
				\verb|A4| & $1$ \\
				\verb|A4| & $1$ \\
				\verb|C7 : C3| & $1$ \\
				\verb|S4| & $1$ \\
				\verb|S4| & $1$ \\
				\verb|PSL(3,2)| & $1$ \\ \hline
			\end{tabular} 
			\end{table}
			
			\begin{table}[H]
				\begin{tabular}{|c|c|}
				\hline
				Subgroups $H$ & $\rho(\psl{2}{8},H)$ \\ \hline \hline
				\verb|1| & $1$ \\
				\verb|C2| & $4$ \\
				\verb|C3| & $1$ \\
				\verb|C2 x C2| & $2$ \\
				\verb|S3| & $\frac{4}{3}$ \\
				\verb|C7| & $\frac{10}{7}$ \\
				\verb|C2 x C2 x C2| & $1$ \\
				\verb|C9| & $1$ \\
				\verb|D7| & $4$ \\
				\verb|D9| & $1$ \\
				\verb|(C2 x C2 x C2) : C7| & $1$ \\
				\verb|PSL(2,8)| & $1$ \\ \hline
			\end{tabular} 
			\end{table}

			\begin{table}[H]
				\begin{tabular}{|c|c|}
				\hline
				Subgroups $H$ & $\rho(\psl{2}{9},H)$ \\ \hline\hline
				\verb|1| & $1$ \\
				\verb|C2| & $2$ \\
				\verb|C3| & $\frac{5}{3}$ \\
				\verb|C3| & $\frac{5}{3}$ \\
				\verb|C2 x C2| & $1$ \\
				\verb|C2 x C2| & $1$ \\
				\verb|C4| & $2$ \\
				\verb|C5| & $\frac{9}{5}$ \\
				\verb|S3| & $\frac{5}{2}$ \\
				\verb|S3| & $\frac{5}{2}$ \\
				\verb|D4| & $1$ \\
				\verb|C3 x C3| & $1$ \\
				\verb|D5| & $\frac{11}{10}$ \\
				\verb|A4| & $\frac{5}{4}$ \\
				\verb|A4| & $\frac{5}{4}$ \\
				\verb|(C3 x C3) : C2| & $1$ \\
				\verb|S4| & $1$ \\
				\verb|S4| & $1$ \\
				\verb|(C3 x C3) : C4| & $1$ \\
				\verb|A5| & $1$ \\
				\verb|A5| & $1$ \\
				\verb|A6| & $1$ \\ \hline
			\end{tabular} 
			\end{table}
			\newpage
			%\onecolumn			
			\begin{table}[H]
				\begin{tabular}{|c|c|}
				\hline
				Subgroups $H$ & $\rho(\psl{2}{11},H)$ \\ \hline \hline
				\verb|1| & $1$ \\
				\verb|C2| & $2$ \\
				\verb|C3| & $\frac{4}{3}$ \\
				\verb|C2 x C2| & $1$ \\
				\verb|C5| & $\frac{12}{5}$ \\
				\verb|S3| & $2$ \\
				\verb|S3| & $2$ \\
				\verb|C6| & $2$ \\
				\verb|D5| & $\frac{17}{10}$ \\
				\verb|C11| & $1$ \\
				\verb|A4| & $1$ \\
				\verb|D6| & $1$ \\
				\verb|C11 : C5| & $1$ \\
				\verb|A5| & $1$ \\
				\verb|A5| & $1$ \\
				\verb|PSL(2,11)| & $1$ \\
				\hline
			\end{tabular} 
			\end{table}
			\begin{table}[H]
				\begin{tabular}{|c|c|}
				\hline
				Subgroups $H$ & $\rho(\psl{2}{13},H)$ \\ \hline \hline
				\verb|1| & $1$ \\
				\verb|C2| & $2$ \\
				\verb|C3| & $\frac{4}{3}$ \\
				\verb|C2 x C2| & $1$ \\
				\verb|C6| & $2$ \\
				\verb|S3| & $2$ \\
				\verb|S3| & $2$ \\
				\verb|C7| & $\frac{9}{7}$ \\
				\verb|A4| & $1$ \\
				\verb|D6| & $1$ \\
				\verb|C13| & $1$ \\
				\verb|D7| & $\frac{10}{7}$ \\
				\verb|D13| & $1$ \\
				\verb|C13 : C3| & $1$ \\
				\verb|C13 : C6| & $1$ \\
				\verb|PSL(2,13)| & $1$ \\ \hline
			\end{tabular} 
			\end{table}
			
			\begin{table}[H]
				\begin{tabular}{|c|c|}
				\hline
				Subgroups $H$ & $\rho(\psl{2}{17},H)$ \\ \hline \hline
				\verb|1| & $1$ \\
				\verb|C2| & $2$ \\
				\verb|C3| & $1$ \\
				\verb|C2 x C2| & $1$ \\
				\verb|C2 x C2| & $1$ \\
				\verb|C4| & $2$ \\
				\verb|S3| & $2$ \\
				\verb|C8| & $2$ \\
				\verb|D4| & $1$ \\
				\verb|D4| & $1$ \\
				\verb|C9| & $\frac{11}{9}$ \\
				\verb|A4| & $1$ \\
				\verb|A4| & $1$ \\
				\verb|D8| & $1$ \\
				\verb|C17| & $1$ \\
				\verb|D9| & $1$ \\
				\verb|S4| & $1$ \\
				\verb|S4| & $1$ \\
				\verb|D17| & $1$ \\
				\verb|C17 : C4| & $1$ \\
				\verb|C17 : C8| & $1$ \\
				\verb|PSL(2,17)| & $1$ \\ \hline
			\end{tabular}
			\end{table}
			\begin{table}[H]
				\begin{tabular}{|c|c|}
					\hline
					Subgroups $H$ & $\rho(\psl{2}{19},H)$ \\ \hline\hline
					\verb|1| & $1$ \\
					\verb|C2| & $2$ \\
					\verb|C3| & $\frac{4}{3}$ \\
					\verb|C2 x C2| & $1$ \\
					\verb|C5| & $\frac{6}{5}$ \\
					\verb|S3| & $2$ \\
					\verb|C9| & $\frac{7}{3}$ \\
					\verb|D5| & $1$ \\
					\verb|D5| & $1$ \\
					\verb|C10| & $2$ \\
					\verb|A4| & $1$ \\
					\verb|D9| & $\frac{4}{3}$ \\
					\verb|C19| & $1$ \\
					\verb|D10| & $1$ \\
					\verb|C19 : C3| & $1$ \\
					\verb|A5| & $1$ \\
					\verb|A5| & $1$ \\
					\verb|C19 : C9| & $1$ \\
					\verb|PSL(2,19)| & $1$ \\ \hline
				\end{tabular}
			\end{table}
		\newpage		
		\onecolumn				
		\end{document}